\documentclass{amsart}

\RequirePackage{setspace}

\usepackage{hyperref}
\RequirePackage{amsmath, amsfonts, amssymb, amsthm}
\RequirePackage[utf8]{inputenc}
\RequirePackage{mathrsfs}
\RequirePackage[english]{babel}
\RequirePackage[T1]{fontenc}
\RequirePackage{mathpple}
\RequirePackage{tikz-cd}
\RequirePackage{mathtools}
\RequirePackage{stmaryrd}
\RequirePackage{xcolor}
\RequirePackage{faktor}
\usepackage{multirow}

\RequirePackage[%
left = \glqq,%
right = \grqq,%
leftsub = \glq,%
rightsub = \grq%
]{dirtytalk}

\newtheorem{thm}{Theorem}[section]

\newtheorem{lem}[thm]{Lemma}

\newtheorem{problem}{Problem}
\newtheorem{claim}{Claim}

\DeclareMathOperator{\tight}{t}

\DeclareMathOperator{\tb}{tb}

\title[]{The Cartesian product of any family of $W$-spaces is $\kappa$-Fr\'{e}chet-Urysohn.}

\begin{document}

\author[M.\ Krupski]{Mikołaj Krupski}
\address{Institute of Mathematics\\ University of Warsaw\\ ul. Banacha 2\\
02--097 Warszawa, Poland }
\email{mkrupski@mimuw.edu.pl}
\author[K.\ Kucharski]{Kacper Kucharski}
\address{Institute of Mathematics\\ University of Warsaw\\ ul. Banacha 2\\
02--097 Warszawa, Poland }
\email{k.kucharski6@uw.edu.pl}

\date{\today}

\begin{abstract}
In this note, we prove that an arbitrary product of $W$-spaces is a $\kappa$-Fr\'{e}chet-Urysohn space. We also show that the boundary tightness of a product $X \times Y$ is at most $\kappa$ provided that $X$ is compact with tightness $t(X) \leq \kappa$ and $Y$ has boundary tightness $\tb(Y) \leq \kappa$. The first result fully resolves, and the second partially addresses, two questions recently raised by Tkachuk.
\end{abstract}




\maketitle

\section{Introduction}

One of the main tasks of general topology is the studying the behavior of topological properties under Cartesian products: given an arbitrary family $\{X_\gamma\}_{\gamma \in \Gamma}$ of spaces possessing a property $\mathcal{P}$, does their product $\prod_{\gamma \in \Gamma} X_\gamma$ satisfy a property $\mathcal{Q}$?\smallskip

Our particular interest lies in the following properties\footnote{For definitions we refer the reader to Section \ref{sec:defin}.} that generalize first-countability.
\smallskip

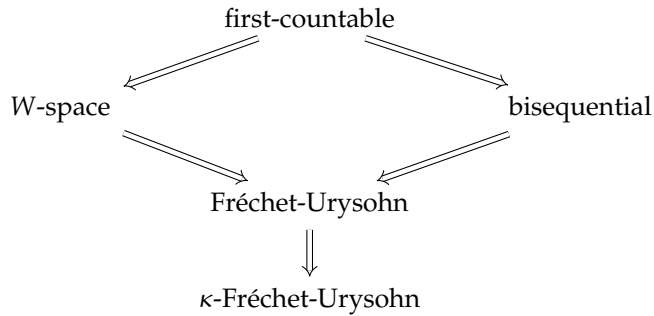
\begin{center}
\begin{tikzcd}
                         & \text{first-countable} \arrow[ld, Rightarrow] \arrow[rd, Rightarrow] &                            \\
W\text{-space} \arrow[rd, Rightarrow] &                                                  & \text{bisequential} \arrow[ld, Rightarrow] \\
                         & \text{Fr\'{e}chet-Urysohn} \arrow[d, Rightarrow]                         &                            \\
                         & \kappa\text{-Fr\'{e}chet-Urysohn}                                        &
\end{tikzcd}
\end{center}

The preservation of these properties under products has been extensively studied. Franklin \cite{F} showed that the Fr\'{e}chet-Urysohn property is not productive, and Simon \cite{S} provided a counterexample even within the class of compact spaces. Turning to weaker properties, Liu and Ludwig proved that while any Cartesian product of bisequential spaces is $\kappa$-Fr\'{e}chet-Urysohn \cite[Theorem 4.1]{LL}, products of Fr\'{e}chet-Urysohn spaces can fail to be $\kappa$-Fr\'{e}chet-Urysohn \cite[Corollary 4.1]{LL}.
\smallskip

This naturally leads to the corresponding question for $W$-spaces, which lie on the other side of the diagram above.
In his recent paper, Tkachuk posed this problem explicitly \cite[Problem 4.8]{T}:

\begin{problem}\label{problem:1}
Let $\{X_{\gamma}\}_{\gamma \in \Gamma}$ be a family of $W$-spaces. Is it true that the product space $\prod_{\gamma \in \Gamma} X_{\gamma}$ is a $\kappa$-Fr\'{e}chet-Urysohn space?
\end{problem}

In Section 3, we resolve Problem \ref{problem:1} in the affirmative (see Theorem \ref{thm:main}).\medskip

A cardinal invariant closely related to first-countability is the tightness $\tight(X)$ of a space $X$. Recently, a natural weakening of this notion, called {\it boundary tightness} of $X$ and denoted $\tb(X)$, was introduced by Tkachuk \cite{T} and studied independently by Gabriyelyan \cite{Ga}.\smallskip

A celebrated theorem of Malykhin \cite{M} states that the product of an arbitrary family of compact spaces $\{X_\gamma:\gamma\in \Gamma\}$ has tightness $\tight(\prod_{\gamma \in \Gamma} X_{\gamma}) \leq \kappa$, provided that $|\Gamma| \leq \kappa$ and $\tight(X_{\gamma}) \leq \kappa$ for all $\gamma \in \Gamma$. It is not clear if an analogous result can be obtained for boundary tightness, even in the simplest case for a product of two spaces (see \cite[Problem 4.7]{T}):

In Section 4, we show that if $X$ is compact with $\tight(X) \leq \kappa$ and $Y$ is a space with $\tb(Y) \leq \kappa$, then the boundary tightness $\tb(X \times Y)$ of the product $X \times Y$ does not exceed $\kappa$ (see Theorem \ref{thm:main2}).

\section{Notation and definitions.}\label{sec:defin}
All spaces considered in this note are Tychonoff topological spaces. As usual, $\omega$ denotes the first infinite cardinal, and is identified with the set of natural numbers $\{0,1,2,\dots\}$. For a set $X$, by $X^{< \omega}$ we denote the set of all finite sequences with values in $X$.\smallskip

Recall that a space $X$ is \textit{Fr\'{e}chet-Urysohn} if for any $A \subset X$ if $x \in \overline{A}$, then there exists a sequence $(x_n)_{n \in \omega} \subset A$ converging to $x$. A space $X$ is called \textit{$\kappa$-Fr\'{e}chet-Urysohn} if it is Fr\'{e}chet-Urysohn with respect to all open subsets of $X$, i.e., if for any open subset $U$ of $X$ if $x\in \overline{U}$, then there exists a sequence $(x_n)_{n \in \omega} \subset U$ converging to $x$.  Here $\kappa$ does not denote any cardinal number and is just a part of the name of the property.\smallskip

The notion of a $W$-space was introduced by Gruenhage \cite{Gr} and it is defined in the language of a certain topological game. Consider the following two Player game on a space $X$ and some point $x \in X$. In the $n$-th round, Player I picks some open neighborhood of $x$, say $U_n$, and Player II responds by choosing a point $x_n \in U_n$. Player I wins
if the sequence $(x_n)_{n \in \omega}$ converges to $x$. Otherwise, Player II wins. This game is called the \textit{$W$-game at $x$} and is denoted by $W(X,x)$.

We follow Gruenhage's terminology \cite{Gr} for the $W$-game.
Let $\tau(X,x)$ denote the family of all nonempty open neighborhoods of $x \in X$.
Since the strategy $\sigma$ of the $W$-game at $x$ for player I depends only on the moves of player II, we can view such a strategy as a function $\sigma:X^{<\omega}\to \tau(X,x)$ that assigns to each finite sequence in $X$ an open neighborhood of $x$. We say that a finite sequence $s=(x_1,\ldots,x_n)\in X^{<\omega}$ is a \textit{$\sigma$-sequence} if $s=\emptyset$ or $x_{1}\in \sigma(\emptyset)$ and $x_{i+1}\in \sigma(x_1,\ldots ,x_i)$ for $i=1,\ldots ,n-1$. An infinite sequence $(x_1,x_2,\ldots)$ is a \textit{$\sigma$-sequence} if all of its initial segments are $\sigma$-sequences. The strategy $\sigma$ is winning if every infinite $\sigma$-sequence converges to $x$. Note that if $s\in X^{<\omega}$ is not a $\sigma$-sequence, then we may define $\sigma(s)$ arbitrarily, say $\sigma(s)=X$.

A point $x \in X$ is called a \textit{$W$-point} if Player I has a winning strategy in the game $W(X,x)$. If all points of $X$ are $W$-points, then such a space is called a {\it $W$-space}.\smallskip

Now, let $\kappa$ be an infinite cardinal number. For a space $X$ and a subspace $A \subset X$ the $\kappa$-closure of $A$ in $X$ is the set
$$
\overline{A}^{\kappa} = \bigcup \{\overline{B} \colon B \subset A \; \text{ and } \; |B| \leq \kappa\}.
$$
Recall that the {\it tightness} of $X$ is the following cardinal number $\tight(X)$:
$$
\tight(X) = \min\{\kappa \colon \forall A \subset X \; (\overline{A} = \overline{A}^{\kappa})\}.
$$
In his recent work \cite{T} Tkachuk considered the following modification of $\tight(X)$ by restricting our attention to open subspaces of a given space $X$:
$$
\tb(X) = \min\{\kappa \colon \forall U \subset X \; \text{ if } U \text{ is open, then } \; (\overline{U} = \overline{U}^{\kappa})\}.
$$

\section{Any product of $W$-spaces is $\kappa$-Fr\'{e}chet-Urysohn.}
In this section, we prove that an arbitrary product of $W$-spaces is $\kappa$-Fr\'{e}chet-Urysohn. To this end, we need the following two lemmas. The first is a recent result of Tkachuk, while the second, due to Gruenhage, is well known.

\begin{lem}\cite[Corollary 3.14]{T}\label{lem:tb_W-spaces}
If $X$ is the product of an arbitrary family of $W$-spaces, then $\tb(X) \leq \omega$.
\end{lem}

\begin{lem}\cite[Theorem 4.1]{Gr}\label{lem:ctbl_prod_W-spaces}
The countable product of $W$-spaces is a $W$-space.
\end{lem}

We are now ready to prove our main result.

\begin{thm}\label{thm:main}
Any product of $W$-spaces is $\kappa$-Fr\'{e}chet-Urysohn.
\end{thm}

\begin{proof}
Let $\{X_{\gamma}\}_{\gamma \in \Gamma}$ be a family of $W$-spaces and let $X$ denote the product $\prod_{\gamma \in \Gamma} X_{\gamma}$. For $\gamma \in \Gamma$ and $A \subset \Gamma$, let $\pi_{\gamma} \colon X \to X_{\gamma}$ and $\pi_A \colon X \to \prod_{\gamma \in A} X_{\gamma}$ be the standard projections.
Let $U$ be a nonempty open subset of $X$, and fix $x \in \overline{U}$. By Lemma~\ref{lem:tb_W-spaces}, there exists a countable set $S \subset U$ with $x \in \overline{S}$. Let $\{s_n \colon n \in \omega\}$ be an enumeration of $S$.
For each $n \in \omega$, find a finite set $F_n \subset \Gamma$, and for each $\gamma \in F_n$ let $U_{\gamma}^n$ be an open subset of $X_{\gamma}$ such that
$$
s_n \in \bigcap_{\gamma \in F_n} \pi_{\gamma}^{-1}(U_{\gamma}^n) \subset U.
$$
In other words, $\bigcap_{\gamma \in F_n} \pi_{\gamma}^{-1}(U_{\gamma}^n)$ is a basic open neighborhood of $s_n$ in $X$ contained in $U$.
Let $A = \bigcup_{n \in \omega} F_n$ and define
$$
Y = \pi_A(X) = \prod_{\gamma \in A} X_{\gamma}.
$$
For $\gamma \in A$ let $p_{\gamma} \colon Y \to X_{\gamma}$ be the standard projection.\smallskip

Set $y = \pi_A(x)$, $t_n = \pi_A(s_n)$, and $V_n = \bigcap_{\gamma \in F_n} p_{\gamma}^{-1}(U_{\gamma}^n)$ for each $n \in \omega$. Finally, let $T = \{t_n \colon n \in \omega\}$ and notice that $y \in \overline{T}$ (where the closure is taken in $Y$).\smallskip

Since $A$ is countable, $Y$ is a $W$-space by Lemma \ref{lem:ctbl_prod_W-spaces}. Let $\sigma$ be a winning strategy for Player I in the $W$-game $W(Y,y)$. By induction, we will construct
\begin{itemize}
    \item a sequence $(n_k)_{k \in \omega}$ of natural numbers and
    \item a convergent sequence $y_n \to y$ of points of $Y$,
\end{itemize}
such that
\begin{equation}\label{eq}
 y_k \in V_{n_k} \;\text{ and }\; y_{k + 1} \in \sigma(y_0, \dots, y_{k}) \; \text{ for all }\; k \in \omega.
\end{equation}

Let $W_0 = \sigma(\emptyset)$. Since $y \in \overline{T}$ and $W_0$ is an open neighborhood of $y$, we have $W_0 \cap T \neq \emptyset$. Hence, we can find $n_0$ such that $t_{n_0} \in W_0$. Since we also have $t_{n_0} \in V_{n_0}$, if $y_0=t_{n_0}$ then \eqref{eq} holds for $y_0$.

Fix $k\in \omega$ and assume that the numbers $n_0, \dots, n_k$ and points $y_0, \dots, y_k$ are already chosen. Consider the set $W_{k + 1} = \sigma(y_0, \dots, y_k)$. Again, since $y \in \overline{T}$ and $W_{k+1}$ is an open neighborhood of $y$, we have $U_{k + 1} \cap T \neq \emptyset$. Hence, there exists $n_{k + 1}$ with $t_{n_{k + 1}} \in W_{k + 1}$. Since $t_{n_{k + 1}} \in V_{n_{k + 1}}$, if $y_{k + 1}=t_{n_{k+1}}$ then \eqref{eq} is satisfied. This finishes the induction.

The strategy $\sigma$ is winning for Player I, so by \eqref{eq}, the sequence $(y_n)_{n \in \omega}$ converges to $y$. Define $x_n\in X$ by
$$
x_n(\gamma)=
\begin{cases}
y_n(\gamma) \quad \text{if } \gamma\in A\\
x(\gamma) \quad \text{if } \gamma\notin A.
\end{cases}
$$
It is readily seen that the
sequence $(x_n)_{n \in \omega}$ converges to $x$. Moreover, since $y_k \in V_{n_k}$ for all $k \in \omega$ (cf. \eqref{eq}), it follows that $x_n \in U$ for all $n \in \omega$. This finishes the proof.
\end{proof}

\section{Boundary tightness and products.}

It is an open question whether the product of two compact spaces $X$ and $Y$ with $\tb(X) \le \kappa$ and $\tb(Y) \le \kappa$ has boundary tightness less than or equal to $\kappa$ (see \cite[Problem 4.7]{T}). However, it turns out that if we assume more about one of the factors, namely that one of the factors has tightness less than or equal to $\kappa$, then $\tb(X \times Y) \le \kappa$. The proof is a slight modification of that of the celebrated theorem of Malykhin \cite{M} (cf. \cite[p. 113]{J}).

\begin{thm}\label{thm:main2}
Let $\kappa$ be an infinite cardinal. If $X$ is compact with $\tight(X)\leq \kappa$ and $Y$ is a Tychonoff space with $\tb(Y) \leq \kappa$, then $\tb(X \times Y) \leq \kappa$.
\end{thm}

\begin{proof}
Fix a nonempty open set $U \subset X \times Y$ and let $A$ denote its $\kappa$-closure $\overline{U}^{\kappa}$. It is suffices to prove that $\overline{U} = A$. Striving for a contradiction, assume that this is not the case and fix $(x_0,y_0) \in \overline{U} \setminus A$. Set
$$
A_{y_0} = \{x \in X \colon (x,y_0) \in A\}
$$
and note that $x_0 \notin A_{y_0}$.

\begin{claim}\label{cl:1}
The set $A_{y_0}$ is closed.
\end{claim}
To see that this is the case, let $x \in \overline{A_{y_0}}$. Since $\tight(X) \leq \kappa$, there is $B \subset A_{y_0}$ of cardinality $|B| \leq \kappa$ such that $x \in \overline{B}$. Now $(x,y_0) \in \overline{B} \times \{y_0\}$ and since $A$ is $\kappa$-closed and $|B \times \{y_0\}| \leq \kappa$, we have $\overline{B} \times \{y_0\} \subset A$. This shows that $(x,y_0) \in A$, whence $x \in A_{y_0}$.\hfill $\blacksquare$\smallskip

Choose open sets $V,W \subset X$ containing $x_0$ and $A_{y_0}$ respectively such that $\overline {V}\cap W=\emptyset$, and define
$$
U' = U \cap (V \times Y) \;\text{ and }\; A' = \overline{U'}^{\kappa}.
$$

\begin{claim}\label{cl:2}
The set $A'_{y_0} = \{x \in X \colon (x,y_0) \in A'\}$ is empty.
\end{claim}
Indeed, on the one hand, we have $A'_{y_0} \subset A_{y_0} \subset W$, while on the other hand, we have $A' \subset \overline{V} \times Y$ so $A'_{y_0} \subset \overline{V}$. Since $\overline{V} \cap W = \emptyset$, the claim follows.\hfill $\blacksquare$\smallskip


\begin{claim}\label{cl:3}
$(x_0,y_0) \in \overline{U'}$.
\end{claim}
To see that this is the case, fix a basic open neighborhood $U_{x_0} \times U_{y_0}$ of $(x_0,y_0)$. Shrinking $U_{x_0}$ if necessary, without loss of generality, we can assume that $U_{x_0} \subset V$. Since $(x_0,y_0) \in \overline{U}$, we have $U \cap (U_{x_0} \times U_{y_0})\neq \emptyset$. But $U \cap (U_{x_0} \times U_{y_0}) = U'\cap (U_{x_0} \times U_{y_0})$, so we are done.\hfill $\blacksquare$\smallskip

Let $\pi_Y \colon X \times Y \to Y$ be the projection onto $Y$. Consider the set
$$
E = \pi_Y(U').
$$
and note that $E$ is an open subset of $Y$. By continuity of $\pi_Y$ and Claim \ref{cl:3} we get $y_0 \in \overline{\pi_Y(U')} = \overline{E}$. Since $\tb(Y) \leq \kappa$, there exists a set $F \subset E$ of cardinality $|F| \leq \kappa$ with $y_0 \in \overline{F}$. For each $y \in F$, choose $x_y \in X$ such that $(x_y,y) \in U'$ and consider the set
$$
G = \{(x_y,y) \colon y \in F\}.
$$
Obviously, $G \subset U'$ and $|G| \leq \kappa$. Thus, $\overline{G} \subset A'$ because $A'$ is the $\kappa$-closure of $U'$. Now, since $X$ is compact, $\pi_Y(\overline{G})$ is a closed subset of $Y$, whence
$$
\pi_Y(G) = F \subset \overline{F} \subset \pi_Y(\overline{G}).
$$
Since $y_0 \in \overline{F}$, we have $y_0 \in \pi_Y(\overline{G})$. It follows that there exists $x \in X$ such that
$$
(x,y_0) \in \overline{G} \subset A'.
$$
This is a contradiction with Claim \ref{cl:2}.
\end{proof}

\bibliographystyle{siam}
\bibliography{bib.bib}

@article {T,
    AUTHOR = {Tkachuk, V. V.},
     TITLE = {Spaces with accessible boundaries of open sets},
   JOURNAL = {Rev. R. Acad. Cienc. Exactas F\'is. Nat. Ser. A Mat. RACSAM},
  FJOURNAL = {Revista de la Real Academia de Ciencias Exactas, F\'isicas y
              Naturales. Serie A. Matematicas. RACSAM},
    VOLUME = {120},
      YEAR = {2026},
    NUMBER = {1},
     PAGES = {Paper No. 26, 9},
      ISSN = {1578-7303,1579-1505},
   MRCLASS = {54A25 (54A20 54C35)},
  MRNUMBER = {5002794},
       DOI = {10.1007/s13398-025-01820-2},
       URL = {https://doi.org/10.1007/s13398-025-01820-2},
}

@article {M,
    AUTHOR = {Malyhin, V. I.},
     TITLE = {The tightness and {S}uslin number in {${\rm exp}\ X$} and in a
              product of spaces},
   JOURNAL = {Dokl. Akad. Nauk SSSR},
  FJOURNAL = {Doklady Akademii Nauk SSSR},
    VOLUME = {203},
      YEAR = {1972},
     PAGES = {1001--1003},
      ISSN = {0002-3264},
   MRCLASS = {54H05 (54B10)},
  MRNUMBER = {300241},
MRREVIEWER = {Peter\ Kessler},
}

@book {J,
    AUTHOR = {Juh\'asz, Istv\'an},
     TITLE = {Cardinal functions in topology---ten years later},
    SERIES = {Mathematical Centre Tracts},
    VOLUME = {123},
   EDITION = {Second},
 PUBLISHER = {Mathematisch Centrum, Amsterdam},
      YEAR = {1980},
     PAGES = {iv+160},
      ISBN = {90-6196-196-3},
   MRCLASS = {54-02 (54A25)},
  MRNUMBER = {576927},
MRREVIEWER = {W.\ W.\ Comfort},
}

@article {LL,
    AUTHOR = {Liu, Chuan and Ludwig, Lewis D.},
     TITLE = {{$\kappa$}-{F}r\'echet {U}rysohn spaces},
   JOURNAL = {Houston J. Math.},
  FJOURNAL = {Houston Journal of Mathematics},
    VOLUME = {31},
      YEAR = {2005},
    NUMBER = {2},
     PAGES = {391--401},
      ISSN = {0362-1588},
   MRCLASS = {54D55 (54E40)},
  MRNUMBER = {2132843},
MRREVIEWER = {Ljubi\v sa\ Ko\v cinac},
}

@article {Gr,
    AUTHOR = {Gruenhage, Gary},
     TITLE = {Infinite games and generalizations of first-countable spaces},
   JOURNAL = {General Topology and Appl.},
  FJOURNAL = {General Topology and its Applications},
    VOLUME = {6},
      YEAR = {1976},
    NUMBER = {3},
     PAGES = {339--352},
      ISSN = {0016-660X},
   MRCLASS = {54D55 (90D45)},
  MRNUMBER = {413049},
MRREVIEWER = {R.\ Telg\'arsky},
}

@article {Ga,
    AUTHOR = {Gabriyelyan, Saak},
     TITLE = {Spaces with open countable tightness},
   JOURNAL = {Eur. J. Math.},
  FJOURNAL = {European Journal of Mathematics},
    VOLUME = {12},
      YEAR = {2026},
    NUMBER = {2},
     PAGES = {Paper No. 22, 18},
      ISSN = {2199-675X,2199-6768},
   MRCLASS = {54D99 (46A03 54A05 54C35)},
  MRNUMBER = {5076689},
       DOI = {10.1007/s40879-026-00900-w},
       URL = {https://doi.org/10.1007/s40879-026-00900-w},
}

@article {F,
    AUTHOR = {Franklin, S. P.},
     TITLE = {Spaces in which sequences suffice},
   JOURNAL = {Fund. Math.},
  FJOURNAL = {Polska Akademia Nauk. Fundamenta Mathematicae},
    VOLUME = {57},
      YEAR = {1965},
     PAGES = {107--115},
      ISSN = {0016-2736,1730-6329},
   MRCLASS = {54.22},
  MRNUMBER = {180954},
MRREVIEWER = {I.\ Namioka},
       DOI = {10.4064/fm-57-1-107-115},
       URL = {https://doi.org/10.4064/fm-57-1-107-115},
}

@article {S,
    AUTHOR = {Simon, Petr},
     TITLE = {A compact {F}r\'echet space whose square is not {F}r\'echet},
   JOURNAL = {Comment. Math. Univ. Carolin.},
  FJOURNAL = {Commentationes Mathematicae Universitatis Carolinae},
    VOLUME = {21},
      YEAR = {1980},
    NUMBER = {4},
     PAGES = {749--753},
      ISSN = {0010-2628,1213-7243},
   MRCLASS = {54D30 (04A20)},
  MRNUMBER = {597764},
MRREVIEWER = {J.\ E.\ Vaughan},
}

\end{document}